\newif\ifappendix\appendixfalse
\newif\iffull\fullfalse 

\makeatletter \@input{flags} \makeatother

\documentclass[a4paper,UKenglish,cleveref, autoref, thm-restate]{lipics-v2021}

\pdfoutput=1 
\hideLIPIcs  

\usepackage{mathpartir}
\usepackage{ifluatex}
\ifluatex{}
\usepackage{fontspec}
\fi{}
\usepackage{amsmath}
\usepackage{amsthm}
\ifluatex{}
\else{}
\usepackage{amssymb}
\fi{}
\usepackage{tikz-cd}
\usepackage[margin=false,inline=true]{fixme}

\FXRegisterAuthor{aaa}{anaaa}{\color{cyan}AAA}

\usepackage{stmaryrd}

\nolinenumbers

\iffull{}
\usepackage[conf={normal}]{proof-at-the-end}
\else
\usepackage[conf={no link to proof, restate}]{proof-at-the-end}
\fi

\EventEditors{J\"{o}rg Endrullis and Sylvain Schmitz}
\EventNoEds{2}
\EventLongTitle{33rd EACSL Annual Conference on Computer Science Logic (CSL 2025)}
\EventShortTitle{CSL 2025}
\EventAcronym{CSL}
\EventYear{2025}
\EventDate{February 10--14, 2025}
\EventLocation{Amsterdam, Netherlands}
\EventLogo{}
\SeriesVolume{326}
\ArticleNo{36}

\title{Kleene algebra with commutativity conditions is undecidable}

\author{Arthur Azevedo de Amorim}{Rochester Institute of Technology}{arthur.aa@gmail.com}{0000-0001-9916-6614}{National Science Foundation Grant No. 2314323.}

\author{Cheng Zhang\footnote{Work performed at Boston Unviersity}}{University College London}{czhang03@bu.edu}{https://orcid.org/0000-0002-8197-6181}{National Science Foundation Award No. 1845803 and No. 2040249.}

\author{Marco Gaboardi}{Boston University}{gaboardi@bu.edu}{0000-0002-5235-7066}{National Science Foundation Award No. 1845803 and No. 2040249.}

\authorrunning{A. A. de Amorim, C. Zhang and M. Gaboardi}

\Copyright{Arthur Azevedo de Amorim, Cheng Zhang and Marco Gaboardi}



\DeclareMathOperator{\Mon}{\mathsf{Mon}}
\DeclareMathOperator{\preKA}{\mathsf{preKA}}
\DeclareMathOperator{\KA}{\mathsf{KA}}

\DeclareMathOperator{\Inc}{\mathsf{Inc}}

\DeclareMathOperator{\If}{\mathsf{If}}
\DeclareMathOperator{\Halt}{\mathsf{Halt}}
\DeclareMathOperator{\Next}{\mathsf{Next}}

\DeclareMathOperator{\Comm}{\mathsf{Comm}}
\DeclareMathOperator{\Set}{\mathsf{Set}}
\DeclareMathOperator{\Rel}{\mathsf{Rel}}

\acknowledgements{We want to thank Todd Schmid and Alexandra Silva for the
  valuable discussion during this work; and also acknowledge the useful feedback
  provided by the anonymous reviewers of CSL2025.  Finally, we want to thank
  Stepan Kuznetsov for his detailed and valuable feedback on this work.}

\begin{document} 

\maketitle

\begin{abstract}
  We prove that the equational theory of Kleene algebra with commutativity
  conditions on primitives (or atomic terms) is undecidable, thereby settling a
  longstanding open question in the theory of Kleene algebra.  While this
  question has also been recently solved independently by Kuznetsov, our results
  hold even for weaker theories that do not support the \emph{induction axioms}
  of Kleene algebra.
\end{abstract}

\section{Introduction}

\emph{Kleene algebra} generalizes the algebra of regular languages while
retaining many of its pleasant properties, such as having a decidable equational
theory.  This enables numerous applications in program verification, by
translating programs and specifications into Kleene-algebra terms and then
checking these terms for equality.  This idea has proved fruitful in many
domains, including networked
systems~\cite{Anderson_Foster_Guha_Jeannin_Kozen_Schlesinger_Walker_2014,Foster_Kozen_Milano_Silva_Thompson_2015},
concurrency~\cite{Hoare_Möller_Struth_Wehrman_2009,Kappé_Brunet_Silva_Wagemaker_Zanasi_2020,Kappé_Brunet_Silva_Zanasi_2018},
probabilistic programming~\cite{McIver_Cohen_Morgan_2006,
  McIver_Rabehaja_Struth_2011}, relational
verification~\cite{Antonopoulos_Koskinen_Le_Nagasamudram_Naumann_Ngo_2022},
program schematology~\cite{Angus_Kozen_2001}, and program
incorrectness~\cite{Zhang_deAmorim_Gaboardi_2022}.

Many applications require extending Kleene algebra with other axioms.  A popular
extension is adding commutativity conditions $e_1e_2 = e_2e_1$, which state that
$e_1$ and $e_2$ can be composed in any order.  In terms of program analysis,
$e_1$ and $e_2$ correspond to commands of a larger program, and commutativity
ensures that their order does not affect the final output.  Such properties have
been proven useful for relational
reasoning~\cite{Antonopoulos_Koskinen_Le_Nagasamudram_Naumann_Ngo_2022} and
concurrency~\cite{diekert_PartialCommutationTraces_1997}.

Unfortunately, such extensions can pose issues for decidability.  In particular,
even the addition of equations of the form $xy = yx$, where $x$ and $y$ are
primitives, can make the equivalence of two \emph{regular languages} given by
Kleene algebra terms~\cite{Kozen_1996,gibbons_DecidabilityProblemsRational_1986,berstel_TransductionsContextFreeLanguages_1979} undecidable---in fact,
$\Pi _1^0$-complete~\cite{DBLP:conf/lics/Kozen97}, or equivalent to the
complement of the halting problem.

Despite this negative result, it was still unknown whether we could decide such
equations in \emph{arbitrary} Kleene algebras---or, equivalently, decide whether
an equation can be derived solely from the Kleene-algebra axioms.  Indeed, since
the set of Kleene-algebra equations is generated by finitely many clauses, it is
recursively enumerable, or $ \Sigma _1^0$.  Since a set cannot be simultaneously
$ \Sigma _1^0$ and $ \Pi _1^0$-complete, the problem of deciding equations under
commutativity conditions for all regular languages is not the same as the
problem of deciding such equations for all Kleene algebras.  There must be
equations that are valid for all regular languages, but not for arbitrary Kleene
algebras.  Nevertheless, the question of decidability of Kleene-algebra
equations with commutativity conditions remained open for almost 30
years~\cite{DBLP:conf/lics/Kozen97}.

This paper settles this question \emph{negatively}, proving that this problem is
undecidable.  In fact, undecidability holds even for weaker notions of Kleene
algebra that do not validate its \emph{induction axioms}, which are needed to
prove many identities involving the iteration operation. At a high level, our
proof works as follows.  Given a machine $M$ and an input $x$, we define an
inequality between Kleene algebra terms with the following two properties: (1)
if $M$ halts on $x$ and accepts, the inequality holds, but (2) if $M$ halts on
$x$ and rejects, the inequality does not hold.  If such inequalities were
decidable, we would be able to computationally distinguish these two scenarios,
which is impossible by diagonalization.

\paragraph*{On Kuznetsov's Undecidability Proof}

As we were finishing this paper, we learned that the question of undecidability
had also been settled independently by Kuznetsov in recent
work~\cite{Kuznetsov23}.  Though our techniques overlap, there are two
noteworthy differences between the two proofs.  On the one hand, Kuznetsov's
proof uses the induction axioms of Kleene algebra, so it applies to fewer
settings. On the other hand, Kuznetsov was able to prove that the equational
theory of Kleene algebra with commutativity conditions is, in fact
$ \Sigma _1^0$-complete, by leveraging \emph{effective inseparability}, a standard
notion of computability theory.  After learning about Kuznetsov's work, we could
adapt his argument to derive completeness in our more general setting as well,
so this paper can be seen as a synthesis of Kuznetsov's work and our own.

\paragraph*{Structure of the paper}
In \Cref{sec:preliminaries}, we recall basic facts about Kleene algebra, and
introduce a framework for stating the problem of equations modulo commutativity
conditions using category theory.

In \Cref{sec:machines}, we present the core of our undecidability proof. We use
algebra terms to model the transition relation of an abstract machine, and
construct a set of inequalities that allows us to tell whether a machine accepts
a given input or not.  If we could decide such inequalities, we would be able to
distinguish two effectively inseparable sets, which would lead to a contradiction.
This argument hinges on a \emph{completeness result} (\Cref{thm:completeness}),
which guarantees that, if a certain machine accepts an input, then a
corresponding inequality holds.

In \Cref{sec:relations}, we prove that an
analog of the completeness result holds for a large class of relations that can
be represented with terms, provided that they satisfy a technical condition that
allows us to reason about the image of a set by a relation.

In \Cref{sec:proving-representability}, we develop techniques to prove that the
machine transition relation satisfies the required technical conditions for
completeness. In \Cref{sec:automata}, we show how we can view Kleene algebra
terms as automata, proving an \emph{expansion lemma} (\Cref{lem:expansion}) that
guarantees that most terms can be expanded so that all of its matched strings
bounded by some maximum length can be identified.  This framework generalizes
the usual definitions of derivative on Kleene algebra terms, but does not rely
on the induction axioms of Kleene algebra.  In \Cref{sec:bounded-output}, we
show how we can refine the expansion lemma when terms have
\emph{bounded-output}, which, roughly speaking, means such terms represent
relations that map a string to only finitely many next strings.  We prove that
the transition relation satisfies these technical conditions
(\Cref{sec:finalizing}), which concludes the undecidability proof.

We conclude in \Cref{sec:conclusion}, providing a detailed comparison between
our work and the independent work of Kuznetsov~\cite{Kuznetsov23}.

\section{Kleene Algebra and Commutable Sets}
\label{sec:preliminaries}

To set the stage for our result, we recall some basic facts about Kleene algebra
and establish some common notation that we will use throughout the paper.  We
also introduce a notion of \emph{commutable set}, which we will use to define
algebras with commutativity conditions.

A (left-biased) \emph{pre-Kleene algebra} is an idempotent semiring $X$
equipped with a star operation. Spelled out explicitly, this means that $X$ has
operations
\begin{align*}
  0 & : X &
  1 & : X \\
  (-)+(-) & : X \times X \to X &
  (-)\cdot(-) & : X \times X \to X &
  (-)^* & : X \to X,
\end{align*}
which are required to satisfy the following equations:
\begin{align*}
  1 \cdot x & = x \cdot 1 = x & 
  0 \cdot x & = x \cdot 0 = 0 &
  x \cdot (y \cdot z) & = (x \cdot y) \cdot z \\
  0 + x & = x & 
  x + y & = y + x &
  x + (y + z) & = (x + y) + z \\
  x \cdot (y + z) & = x \cdot y + x \cdot z & 
  (x + y) \cdot z & = x \cdot z + y \cdot z &
  x^* & = 1 + x \cdot x^*,
\end{align*}
where the last rule \(x^* = 1 + x \cdot x^*\) is named ``left unfolding''.  A
pre-Kleene algebra carries the usual ordering relation on idempotent monoids:
$x \leq y$ means that $y + x = x$.  A \emph{Kleene algebra} is a pre-Kleene
algebra that satisfies the following properties:
\begin{align*}
  xy \le y & \Rightarrow x^*y \le y & xy \le x & \Rightarrow xy^* \le x,
\end{align*}
dubbed left and right induction.  A \emph{$*$-continuous Kleene algebra} is a
pre-Kleene algebra where, for all $p$, $q$ and $r$, $\sup_{n \geq 0}pq^nr$
exists and is equal to $pq^*r$.  Every $*$-continuous algebra satisfies the
induction axioms, so it is, in fact, a Kleene algebra.

\begin{example}\label{exp:preKA-that-is-not-a-WKA}
  Though many pre-Kleene algebras that we'll consider are actually proper Kleene
  algebras, the two theories do not coincide.  The following algebra, adapted
  from Kozen~\cite{kozen_KleeneAlgebrasClosed_1990}, validates all the
  pre-Kleene algebra axioms, but not the induction ones. The carrier set of the
  algebra is \(\mathbb{N} + \{\bot, \top\}\), ordering by posing
  \( \bot \leq n \leq \top \) for all \(n \in \mathbb{N}\).  The addition
  operation computes the maximum of two elements. Multiplication is defined as:
  \begin{align*}
    x \cdot  \bot  & \triangleq  \bot  \cdot x \triangleq  \bot  &
    x \cdot  \top  & \triangleq  \top  \cdot x \triangleq  \top  \text{ when } x \neq \bot &
    x \cdot y & \triangleq x +_{\mathbb{N}} y
  \end{align*}
  where \(+_{\mathbb{N}}\) is the usual addition operation on natural numbers.
  The neutral elements of addition and multiplication are respectively $\bot$ and \(0\).  
  The star operation is defined as follows:
  \[x^* \triangleq
    \begin{cases}
      0 & \text{if } x = \bot \\
      \top & \text{otherwise}
    \end{cases}
  \]
  We can verify the unfolding rule by case analysis:
  \begin{align*}
    \text{when \(x = \bot\):}
    \quad&
           \bot^*
           = 0
           = \max(0, \bot)
           = \max(0, \bot \cdot 0)
           = \max(0, \bot \cdot \bot^*);\\
    \text{when \(x \neq \bot\):}
    \quad&
           x^*
           = \top
           = \max(0, \top)
           = \max(0, x \cdot \top)
           = \max(0, x \cdot x^*).
  \end{align*}
  However, this algebra is not a proper Kleene algebra, because it doesn't
  satisfy \((x^*)^* = x^*\) (which must hold in any Kleene algebra). Indeed,
  \((\bot^*)^* = 0^* = \top \neq 0 = \bot^*.\)
\end{example}

\begin{remark}
  Weak Kleene algebras~\cite{kozen_LefthandedCompleteness_2020} are algebraic
  structures that sit between proper Kleene algebras and pre-Kleene algebras.
  They need not validate the induction axioms of Kleene algebra, but satisfy
  more rules than just left unfolding---in particular, $(x^*)^* = x^*$.  Thus,
  \Cref{exp:preKA-that-is-not-a-WKA} also shows that the theory of pre-Kleene
  algebras is strictly weaker than that of weak Kleene algebras.
\end{remark}

Let $X$ and $Y$ be pre-Kleene algebras.  A \emph{morphism} of type $X \to Y$ is
a function $f : X \to Y$ that commutes with all the operations.  This gives rise
to a series of categories $\KA^* \subset \KA \subset \preKA$ of $*$-continuous
algebras, Kleene algebras, and pre-Kleene algebras. Each category is a strict
full subcategory of the next one---strict because some pre-Kleene algebras are
not Kleene algebras (\Cref{exp:preKA-that-is-not-a-WKA}) and because some Kleene
algebras are not $*$-continuous~\cite{kozen_KleeneAlgebrasClosed_1990}.

The prototypical example of Kleene algebra is given by the set $\mathcal{L}X$ of
regular languages over some alphabet $X$.  In program analysis applications, a
regular language describes the possible traces of events performed by some
system. We use the multiplication operation to represent the sequential
composition of two systems: if two components produce traces $t_1$ and $t_2$,
then their sequential composition produces the concatenated trace $t_1t_2$,
indicating that the actions of the first component happen first.  Thus, by
checking if two regular languages are equal, we can assert that the behaviors of
two programs coincide.  When $X$ is empty, $\mathcal{L}X$ is isomorphic to the
booleans $\mathbb{2}  \triangleq  \{0 \leq 1\}$.  The addition operation is disjunction, the
multiplication operation is conjunction, and the star operation always outputs
1. This Kleene algebra is the initial object in all three categories $\preKA$,
$\KA$ and $\KA^*$.

The induction property of Kleene algebra allows us to derive several useful
properties for terms involving the star operation.  For example, they imply that
the star operation is monotonic, a \emph{right-unfolding rule} $x^* = 1 + x^*x$,
and also that $x^*x^* = x^*$.  This means that many of intuitions about regular
languages carry over to Kleene algebra.  Unfortunately, when working with
pre-Kleene algebras, most of these results cannot be directly applied, making
reasoning trickier.  In practice, we can only reason about properties of the
star operation that involve a finite number of uses of the left-unfolding rule.
Dealing with this limitation is at the heart of the challenges we will face when
proving our undecidability result.

\subsection{Commuting conditions}

Sometimes, we would like to reason about a system where two actions can be
reordered without affecting its behavior.  For example, we might want to say
that a program can perform assignments to separate variables in any order, or
that actions of separate threads can be executed concurrently.  To model this,
we can work with algebra terms where some elements can be composed in any order.
As we will see, unfortunately, adding such hypotheses indiscriminately can lead
to algebras with an undecidable equational theory.  The notion of
\emph{commutable set}, which we introduce next, allows us to discuss such
hypotheses in generality.

\begin{definition}
  \label{def:comm}
  A \emph{commuting relation} on a set $X$ is a reflexive symmetric relation.  A
  \emph{commutable set} is a carrier set endowed with a commuting relation
  $ \sim $.  We say that two elements $x$ and $y$ commute if $x \sim y$.  A
  commutable set is \emph{commutative} if all elements commute; it is
  \emph{discrete} if the commuting relation is equality.  A morphism of
  commutable sets is a function that preserves the commuting relation, which
  leads to a category $\Comm$.  A \emph{commutable subset} of a commutable set
  $X$ is a commutable set $Y$ whose carrier is a subset of $X$, and whose
  commuting relation is the restriction of $\sim$ to $Y$. %
  {}%
  We'll often abuse notation and treat a subobject $Y \hookrightarrow X$ as a
  commutable subset if its image in $X$ is a commutable subset.
\end{definition}

Given a commutable set, we have various ways of building algebraic structures,
which can be summarized in the diagram of \Cref{fig:comm-set-diagram} (which is
commutative, except for the dashed arrows).
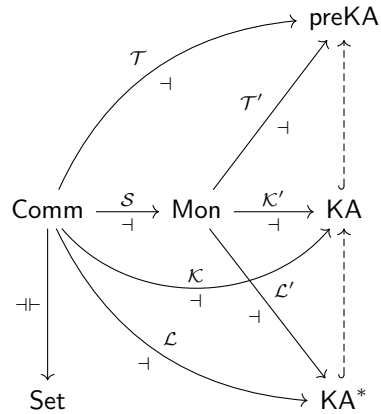
\begin{figure}
\begin{center}
  \begin{tikzcd}[row sep=2cm]
    &
    & \preKA \\
    \Comm \ar[r, "\dashv"', "\mathcal{S}"]
    \ar[urr, "\dashv"',"\mathcal{T}", bend left]
    \ar[drr, "\dashv"',"\mathcal{L}", bend right]
    \ar[rr, "\dashv"', "\mathcal{K}", bend right=50]
    \ar[d, "\dashv \vdash"']
    & \Mon
    \ar[ur, "\dashv"', "\mathcal{T}'"]
    \ar[r, "\dashv"', "\mathcal{K}'"]
    \ar[dr, "\dashv"', "\mathcal{L}'"]
    & \KA \ar[u, hook, dashed] \\
    \Set
    &
    & \KA^* \ar[u, hook, dashed]
  \end{tikzcd}
\end{center}
\caption{Algebraic constructions on commutable sets}
\label{fig:comm-set-diagram}
\end{figure}
The right-pointing arrows, marked with a $ \dashv $, denote free constructions,
in the sense that they have right adjoints that forget structure.  The first
construction, $\mathcal{S}$, is a functor from $\Comm$ to the category $\Mon$ of
monoids and monoid morphisms.  It maps a commutable set $X$ to the monoid
$\mathcal{S}X$ of strings over $X$, where we equate two strings if they can be
obtained from each other by swapping adjacent elements that commute in $X$.  The
monoid operation is string concatenation, and the neutral element is the empty
string.  The corresponding right adjoint views a monoid $Y$ as a commutable set
where $x \sim y$ if and only if $xy = yx$.

Another group of constructions extends a monoid $X$ with the other Kleene
algebra operations, and quotient the resulting terms by the equations we desire.
For example, the elements of $\mathcal{T}'X$ are terms formed with Kleene
algebra operations, where we identify the monoid operation with the
multiplication operation of the pre-Kleene algebra, and where we identify two
terms if they can be obtained from each other by applying the pre-Kleene algebra
equations.  The construction $\mathcal{K}'$ is obtained by imposing further
equations on terms, while $\mathcal{L}'$ is given by the algebra of regular
languages over a monoid~\cite{DBLP:conf/lics/Kozen97}, which we'll define soon.
The right adjoints of these constructions view an algebra as a multiplicative
monoid.  By composing these constructions with $\mathcal{S}$, we obtain free
constructions $\mathcal{T}$, $\mathcal{K}$ and $\mathcal{L}$ that turn any
commutable set into some Kleene-algebra-like structure.

%

Being a free construction means, in particular, that we can embed the elements
of a commutable set $X$ into $\mathcal{S}X$, $\mathcal{T}X$, $\mathcal{K}X$ and
$\mathcal{L}X$, as depicted in this commutative diagram:
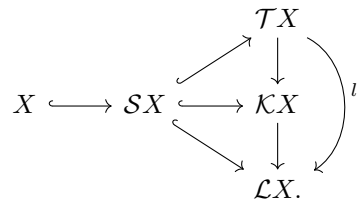
\begin{center}
  \begin{tikzcd}
    &
    & \mathcal{T}X \ar[d] \ar[dd, bend left=60, "l"]\\
    X \ar[r, hook]
    &
    \mathcal{S}X
    \ar[ur, hook]
    \ar[r, hook]
    \ar[dr, hook]
    & \mathcal{K}X \ar{d} \\
    &
    & \mathcal{L}X.
  \end{tikzcd}
\end{center}
By abuse of notation, we'll usually treat $X$ as a proper subset of the free  
algebras. The vertical arrows take the elements of some algebra and impose the
additional identities required by a stronger algebra.  The composite $l$
computes the \emph{language interpretation} of a term, and will play an
important role in our development, as we will see.

Free constructions, like \(\mathcal{T}\), also allow us to define a morphism out
of an algebra $\mathcal{T}X$ simply by specifying how the morphism acts on $X$.
In other words, if $f : X  \to  Y$ is a morphism mapping a commutable set $X$ to a
pre-Kleene algebra $Y$, there exists a unique algebra morphism
$\hat{f} : \mathcal{T}X  \to  Y$ such that the following diagram commutes:
\begin{center}
  \begin{tikzcd}
    \mathcal{T}X \ar[r, "\hat{f}", dashed] & Y \\
    X \ar[u, hook] \ar[ur, "f", swap] &
  \end{tikzcd}
\end{center}
Since $f$ and $\hat{f}$ correspond uniquely to each other, we will not bother
distinguishing between the two.  We'll employ similar conventions for other left
adjoints such as $\mathcal{S}$ or $\mathcal{L}$.

The last construction on \Cref{fig:comm-set-diagram} allow us to turn any
commutable set into a plain set by forgetting its commuting relation.  This
construction has both a left and a right adjoint: the right adjoint views a set
as a commutative commutable set, by endowing it with the total relation; the
left adjoint views a set as a discrete commutative set, by endowing it with the
equality relation.  By turning a set into a commutable set, discrete or
commutative, and then building an algebra on top of that commutable set, we are
able to express the usual notions of free algebra over a set, or of a free
algebra where all symbols are allowed to commute.

\begin{remark}[Embedding algebras]
\label{rem:embeddings}
We introduce some notation for embedding algebras into larger ones.  Suppose
that $X$ is a commutable set and $Y  \subseteq  X$ is a commutable subset.  By
functoriality, this inclusion gives rise to morphisms of algebras of types
$\mathcal{S}Y  \to  \mathcal{S}X$, $\mathcal{T}Y  \to  \mathcal{T}X$, etc. These
morphisms are all injective, because they can be inverted: we can define a
projection $ \pi _Y$ that maps $x  \in  X$ to itself, if $x  \in  Y$, or to 1, if $x  \notin  Y$.
This definition is valid because, since $Y$ inherits the commuting relation from
$X$, and since 1 commutes with everything in $\mathcal{S}Y$, $\mathcal{T}Y$,
etc., we can check that the commuting relation in $X$ is preserved.
\end{remark}

\subsection{Regular Languages}

If $X$ is a monoid, we can view its power set $\mathcal{P}X$ as a $*$-continuous
algebra equipped with the following operations:
\begin{align*}
  0 &  \triangleq   \emptyset  &
  1 &  \triangleq  \{1\} \\
  A + B &  \triangleq  A  \cup  B &
  A  \cdot  B &  \triangleq  \{ xy  \mid  x  \in  A, y  \in  B \} &
  A^* &  \triangleq   \bigcup _{n  \in  \mathbb{N}} A^n.
\end{align*}
The $*$-continuous algebra $\mathcal{L}'X$ of regular languages over $X$ is the
smallest subalgebra of $\mathcal{P}X$ that contains the singletons.  The
language interpretation $l : \mathcal{T}X \to \mathcal{L}X$ is the morphism that
maps a symbol $x \in X$ to the singleton set $\{x\}$.  This allows us to view a
term as a set of strings over $X$, and we will often do this to simplify the
notation; for example, if $e$ is a term, we'll write $X \subseteq e$ to mean
$X \subseteq l(e)$.  Indeed, as the next few results show, it is often safe for
us to view a term as a set of strings.

\begin{theoremEnd}{theorem}
  \label{thm:membership-characterization}
  If $s \in \mathcal{S}X$ is a string and $e \in \mathcal{T}X$ a term, then
  $s \leq e$ is equivalent to $s \in l(e)$.
\end{theoremEnd}

\begin{proofEnd}
  Suppose that $s  \leq  e$.  Then $s  \in  \{s\} = l_X(s)  \subseteq  l_X(e)$ by monotonicity.

  Conversely, suppose that $s  \in  l_X(e)$. We proceed by induction on $e$.
  \begin{itemize}
  \item If $e = x  \in  X$, then $s  \in  l_X(x)$ means that $s = x$.  Thus, we get
    $s  \leq  e$.
  \item If $e = 0$, we get a contradiction.
  \item If $e = 1$, we must have $s = 1$, thus $s  \leq  e$.
  \item If $e = e_{1}e_{2}$, we must have $s = s_{1}s_{2}$, with $s_i  \in  l_X(e_i)$.  By the
    induction hypotheses, $s_i  \leq  e_i$, and thus $s  \leq  e$.
  \item If $e = e_{1} + e_{2}$, then there is some $i$ such that $s  \in  l_X(e_i)$.  By
    the induction hypothesis, $s  \leq  e_i$, and thus $s  \leq  e_{1} + e_{2}$.
  \item Finally, suppose that $e = e_{1}^*$.  Thus, there exists some n such that
    $s  \in  l_x(e_{1})^n$.  This means that we can find a family
    $(s_i)_{i  \in  \{1, \ldots ,n\}}$ such that $s =  \prod _i s_i$ and $s_i  \in  l_x(e_{1})$ for
    every $i$.  By the induction hypothesis, $s_i  \leq  e_{1}$ for every $i$.
    Therefore, $s =  \prod _i s_i  \leq  e_{1}^n  \leq  e_{1}^* = e$.
  \end{itemize}
\end{proofEnd}

\begin{theoremEnd}{theorem}
  \label{thm:finite-term}
  We say that $e  \in  \mathcal{T}X$ is \emph{finite} if its language $l(e)$ is. In
  this case, then $e =  \sum  l(e)$.
\end{theoremEnd}

\begin{proofEnd}
  By induction on $e$. We note that, if $l(e)$ is finite, then $l(e')$ is also
  finite for every immediate subterm $e'$, which allows us to apply the relevant
  induction hypotheses.  If $e$ is of the form $e_1e_2$ and $l(e) =  \emptyset $, this
  need not be the case, but at least one of the factors $e_i$ satisfies
  $l(e_i) =  \emptyset $, which is good enough.
\end{proofEnd}

\begin{theoremEnd}{corollary}
  \label{thm:finiteness-injective}
  The language interpretation $l$ is injective on finite terms: if $l(e_{1}) =
  l(e_{2})$ and both \(e_1\) and \(e_2\) are finite, then $e_{1} = e_{2}$.
\end{theoremEnd}

\begin{proofEnd}
  We have $e_{1} =  \sum l(e_{1}) =  \sum l(e_{2}) = e_{2}$.
\end{proofEnd}

These results allow us to unambiguously view a finite set of strings over $X$ as
a finite term over $X$.  We'll extend this convention to other sets: $Y$ is a
(pre-)Kleene algebra, we are going to view a finite set of elements $A  \subseteq  Y$ as
the element $ \sum _{a  \in  A} a  \in  Y$.

\begin{theoremEnd}{corollary}
  \label{thm:emptiness-characterization}
  For every term $e  \neq  0$, there exists some string $s$ such that $s  \leq  e$.
\end{theoremEnd}

\begin{proofEnd}
  Note that $l(e)  \neq   \emptyset $.  Indeed, if $l(e) =  \emptyset  = l(0)$, then $e = 0$ by
  \Cref{thm:finiteness-injective}, which contradicts our hypothesis.  Therefore,
  we can find some $s$ such that $s  \in  l(e)$.  But this is equivalent to $s  \leq  e$
  by \Cref{thm:membership-characterization}.
\end{proofEnd}

One useful property of Kleene algebra is that, if $X$ is finite, then
$X^*  \in  \mathcal{K}X$ is the top element of the algebra.  This result is
generally not valid for $\mathcal{T}X$, but the following property will be good
enough for our purposes.

\begin{theoremEnd}[normal]{theorem}
  If $X$ is finite and $e  \in  \mathcal{T}X$ is finite, then $eX^*  \leq  X^*$.
\end{theoremEnd}

To conclude our analogy between languages and terms, as far as $*$-continuous
algebras are concerned, elements of $\mathcal{T}X$ are just as good as their
%
corresponding languages---if $Y$ is $*$-continuous, then every morphism of
algebras $f : \mathcal{T}X  \to  Y$ can be factored through the language
interpretation $l$:
\begin{center}
  \begin{tikzcd}
    X \ar[d, hook] \ar[r, hook] & \mathcal{L}X \ar[d] \\
    \mathcal{T}X \ar[r, swap, "f"] \ar[ur, "l"] & Y.
  \end{tikzcd}
\end{center}
This has some pleasant consequences.  For example, let
$[-]_0 : \mathcal{T}X  \to  \mathbb{2}$ be the morphism that maps every $x  \in  X$ to
0. Then $[e]_0 = 1$ if and only if $1  \leq  e$.  Indeed, this morphism must factor
through $\mathcal{L}X$.  The corresponding factoring $\mathcal{L}X$ must map any
nonempty string to 0 and the empty string to 1.  Thus, $[e]_0 = 1$ if and only
if $1  \in  l(e)$, which is equivalent to $1  \leq  e$.

\section{Undecidability via Effective Inseparability}
\label{sec:machines}

Our undecidability result works by using pre-Kleene algebra equations to encode
the execution of \emph{two-counter machines}.  Roughly speaking, a two-counter
machine $M$ is an automaton that has a control state and two counters.  The
machine can increment each counter, test if their values are zero, and halt.
Two-counter and Turing machines are equivalent in expressive power: any
two-counter machine can simulate the execution of a Turing machine, and vice
versa; see Hopcroft et al.~\cite[§8.5.3,§8.5.4]{Hopcroft2001} for an idea of how
this simulation works.  In particular, given a Turing machine $M$, there exists
a two-counter machine that halts on every input where $M$ halts, and yields the
same output for that input.  For this reason, we'll tacitly use two-counter
machines to implement computable functions in what follows.

\begin{definition}
  \label{def:two-counter-machine}
  A \emph{two-counter} machine is a tuple $M = (Q_M, \dot{q},  \iota )$, where $Q_M$
  is a finite set of \emph{control states}, $\dot{q}  \in  Q_M$ is an initial state,
  and $ \iota  : Q_M  \to  I_M$ is a transition function.  The set $I_M$ is the set of
  \emph{instructions} of the machine, defined by the grammar
  \begin{align*}
    I_M \ni i
    & := \Inc(r, q) \mid \If(r,q,q) \mid \Halt(x)
    & (r \in \{1,2\}, q \in Q_M, x \in \{0,1\}).
  \end{align*}
\end{definition}

Two-counter machines act on configurations, which are strings of the form
$a^nb^mq$, where $q$ is a control state and $a$ and $b$ are counter symbols: the
number of symbol occurrences determines which number is stored in a counter.
When the machine halts, it outputs either $1$ or $0$ to indicate whether its
input was accepted or rejected.

\begin{definition}
  \label{def:configurations}
  Let $M$ be a two-counter machine.  We define the following discrete commutable
  sets and terms:
  \begin{align*}
     \Sigma _M &  \triangleq  Q_M + \{a,b,c_0,c_1\} & \text{symbols} \\
    \mathcal{T} \Sigma _M  \ni  C_M &  \triangleq  a^*b^*Q_M & \text{running configurations} \\
    \mathcal{T} \Sigma _M  \ni  T_M &  \triangleq  C_M + \{c_0,c_1\} & \text{all configurations}.
  \end{align*}
\end{definition}

Normally, we would define the semantics of a two-counter machine directly, as a
relation on configurations.  However, it'll be more convenient to instead define
the semantics through algebra terms that describe the graph of this relation,
since we'll use these terms to analyze the execution of a machine with equations
in pre-Kleene algebra. Our definition relies on the following construction.

\begin{definition}
  Let $X$ and $Y$ be commutable sets.  We define a commutable set
  \[ X  \oplus  Y  \triangleq  \{ x_l  \mid  x  \in  X \}  \uplus  \{y_r  \mid  y  \in  Y \}, \] where the commuting
  relation on $X  \oplus  Y$ is generated by the following rules:
  \begin{mathpar}
    \inferrule
    { }
    { x_l  \sim  y_r }
    \and
    \inferrule
    { x  \sim  x' }
    { x_l  \sim  x'_l }
    \and
    \inferrule
    { y  \sim  y' }
    { y_r  \sim  y'_r }
  \end{mathpar}
  The canonical injections $(-)_l : X \to X \oplus Y$ and
  $(-)_r : Y \to X \oplus Y$ are morphisms in $\Comm$ (and present commutable
  subsets).  We abbreviate $X \oplus X$ as $\ddot{X}$.
\end{definition}

If $X$ and $Y$ are commutable sets, we abuse notation and view the functions
$(-)_l : X  \to  X  \oplus  Y$ and $(-)_r : Y  \to  X  \oplus  Y$ as having types
$\mathcal{T}X  \to  \mathcal{T}(X  \oplus  Y)$ and $\mathcal{T}Y  \to  \mathcal{T}(X  \oplus  Y)$.  We
have the corresponding projection functions
$ \pi _l : \mathcal{T}(X  \oplus  Y)  \to  \mathcal{T}X$ and
$ \pi _r : \mathcal{T}(X  \oplus  Y)  \to  \mathcal{T}Y$, where $ \pi _l(y_r) = 1$ for $y  \in  Y$, and
similarly for $ \pi _r$ (cf. \Cref{rem:embeddings}).  If $X$ is a commutable set,
view a term $e  \in  \mathcal{T}X$ as an element $\mathcal{T}\ddot{X}$ by mapping
each symbol $x  \in  X$ in $e$ to $x_lx_r$.  We'll use a similar convention for
strings $\mathcal{S}$.

The idea behind this construction is that any string over $X \oplus Y$ can be
seen as a pair of strings over $X$ and $Y$.  More precisely, the monoids
$\mathcal{S}(X \oplus Y)$ and $\mathcal{S}X \times \mathcal{S}Y$ are isomorphic
via the mappings
\begin{align*}
  \mathcal{S}(X  \oplus  Y)  \ni  s &  \mapsto  ( \pi _l(s),  \pi _r(s))  &
  \mathcal{S}X  \times  \mathcal{S}Y  \ni  (s_1,s_2) &  \mapsto  (s_1)_l(s_2)_r.
\end{align*}
Since a term $e$ over $X \oplus Y$ can be seen as a set of strings over
$X \oplus Y$, we can also view it as a set of pairs of strings over $X$ and
$Y$---in other words, as a relation from $\mathcal{S}X$ to $\mathcal{S}Y$.  We
write $s \to _e s'$ if two strings are related in this way; that is, if
$s_ls'_r \leq e$.

\begin{definition}[Running a two-counter machine]
  \label{def:two-counter-transition}
  We interpret each instruction $i  \in  I_M$ as an element $ \llbracket i \rrbracket   \in  \mathcal{T}\ddot{ \Sigma }_M$:
  \begin{align*}
     \llbracket \Inc(1,q) \rrbracket  &  \triangleq  a_ra^*b^*q_r &
     \llbracket \If(1,q_{1},q_{2}) \rrbracket  &  \triangleq  b^*(q_{1})_r + a_la^*b^*(q_{2})_r \\
     \llbracket \Inc(2,q) \rrbracket  &  \triangleq  a^*b_rb^*q_r &
     \llbracket \If(2,q_{1},q_{2}) \rrbracket  &  \triangleq  a^*(q_{1})_r + a^*b_lb^*(q_{2})_r \\
     \llbracket \Halt(x) \rrbracket  &  \triangleq  (c_x)_r.
  \end{align*}
  The \emph{transition relation} of $M$, $R_M  \in  \mathcal{T}\ddot{ \Sigma }_M$, is defined as
  \begin{align*}
     R_M &  \triangleq   \sum \{ \llbracket  \iota (q) \rrbracket q_l \mid q  \in  Q_M \}.
  \end{align*}
  We say that $M$ \emph{halts} on $n$ if $a^nb^0\dot{q}  \to _{R_M}^* c_x$ for some
  $x  \in  \{0,1\}$.  We refer to $x$ as the \emph{output} of $M$ on $n$.
\end{definition}

\begin{theoremEnd}{lemma}
  The relation $R_M$ satisfies the following property: for every
  $s \to _{R_M} s'$, $s$ is of the form $a^nb^mq \leq C_M$.  Moreover, for any
  $s$ of this form, we have $s' = \llbracket \iota (q) \rrbracket_f (n,m)$, where the
  function $ \llbracket i \rrbracket_f : \mathbb{N} \times \mathbb{N} \to T_M$ is
  defined as follows:
  \begin{align*}
     \llbracket \Inc(1,q) \rrbracket_f (n, m) &  \triangleq  a^{n+1}b^mq &
     \llbracket \If(1,q_{1},q_{2}) \rrbracket_f (n, m) &  \triangleq  \begin{cases}
      a^nb^mq_{1} & \text{if $n = 0$} \\
      a^pb^mq_{2} & \text{if $n = p+1$}
    \end{cases} \\
     \llbracket \Inc(2,q) \rrbracket_f (n, m) &  \triangleq  a^nb^{m+1}q &
     \llbracket \If(2,q_{1},q_{2}) \rrbracket_f (n, m) &  \triangleq  \begin{cases}
      a^nb^mq_{1} & \text{if $m = 0$} \\
      a^nb^pq_{2} & \text{if $m = p+1$}
    \end{cases} \\
     \llbracket \Halt(x) \rrbracket_f (n,m) &  \triangleq  c_x.
  \end{align*}
  In particular, $R_M$ defines a (partial) functional relation on $T_M$.
\end{theoremEnd}

This means that \Cref{def:two-counter-transition} accurately describes the
standard semantics of two-counter machines~\cite{Hopcroft2001}, which allows us
to analyze their properties algebraically.  Combining this encoding with
\Cref{thm:basic-inseparability}, we can show that KA inequalities over
\(\ddot{ \Sigma }_M\) cannot be decided.  More precisely, in the remainder of
the paper, our aim is to prove the following results:

\begin{theoremEnd}{theorem}[Soundness]
  \label{thm:soundness}
  Given a two-counter machine $M$ and a configuration $s  \leq  T_M$, suppose that
  the following inequality holds in $\mathcal{L}\ddot{ \Sigma }_M$:
  \[ s_rR_M^*  \leq   \Sigma ^*(C_M + c_1)_r +  \Sigma _M^* \Sigma _M^{ \neq }\ddot{ \Sigma }_M^*, \] where
  \(
     \Sigma ^{ \neq }_M  \triangleq   \sum _{\substack{x, y  \in   \Sigma  \\ x  \neq  y}} x_ly_r.
  \)
  If $s  \to _{R_M}^* c_x$, then $x = 1$.
\end{theoremEnd}

\begin{proofEnd}
  Suppose that we have some finite sequence of transitions
  $s = s_0  \to   \cdots   \to  s_n = c_x$.  By definition, $(s_i)_l(s_{i+1})_r  \leq  R_M$ for
  every $i  \in  \{0, \ldots ,n-1\}$.  Thus, we have the following inequality on languages:
  \begin{align*}
    p
    &  \triangleq  (s_0)_r  \cdot  (s_0)_l(s_1)_r  \cdots  (s_{n-1})_l(s_n)_r \\
    &  \leq  (s_0)_r  \cdot  R_M  \cdot   \cdots   \cdot  R_M \\
    &  \leq  (s_0)_rR_M^* \\
    &  \leq   \Sigma _M^*(C_M + c_1)_r +  \Sigma _M^* \Sigma _M^{ \neq }\ddot{ \Sigma }_M^*.
  \end{align*}
  On the other hand, by shuffling left and right characters,
  \begin{align*}
    p
    & = (s_0)_r  \cdot  (s_0)_l(s_1)_r  \cdots  (s_{n-1})_l(s_n)_r \\
    & = (s_0)_r(s_0)_l  \cdot  (s_1)_r(s_1)_l  \cdots  (s_{n-1})_r(s_{n-1})_l  \cdot  (s_n)_r \\
    & = s_0  \cdots  s_{n-1} (s_{n})_r \\
    &  \leq   \Sigma _M^*( \Sigma _M)_r^+.
  \end{align*}
  We can check that the languages $ \Sigma _M^*( \Sigma _M)_r^+$ and
  $ \Sigma _M^* \Sigma _M^{ \neq }\ddot{ \Sigma }_M^*$ are disjoint.  Therefore, it must be the case that
  $p  \leq   \Sigma _M^*(C_M + c_1)_r$.  By projecting out the right components, we find that
  $ \pi _r(p) = s_0  \cdots  s_n  \leq   \Sigma _M^*(C_M + c_1)_r$.  We cannot have
  $ \pi _r(p)  \leq   \Sigma _M^*C_M$, since the last character $c_x$ cannot appear in a string
  in $C_M$.  Therefore, $ \pi _r(p)  \leq   \Sigma _M^*c_1$, from which we conclude.
\end{proofEnd}

\begin{theorem}[Completeness]
  \label{thm:completeness}
  Given a two-counter machine $M$ and some configuration $s  \leq  T_M$, we can
  compute a term $ \rho $ with the following property.  If $s  \to _{R_M}^* c_1$, then
  the following inequality is valid in \emph{pre}-Kleene algebra:
  \( sR_M^*  \leq   \Sigma _M^*(C_M + c_1)_r +  \Sigma _M^* \Sigma _M^{ \neq } \rho . \)
\end{theorem}

The soundness theorem can be shown by establishing a correspondence between
traces of two-counter machines and the languages, then arguing that the language
inequality implies that \(M\) halts and outputs 1.  However, an inequality
between terms is always stronger than the same inequality on languages.  Thus,
for completeness, we need to establish a stronger inequality between terms.

To obtain undecidability from soundness and completeness, we leverage
\emph{effective inseparability}, a notion from computability theory.  In what
follows, we use the notation $ \langle x \rangle $ to refer to some effective
encoding of the object $x$ as a natural number.\footnote{%
  Note that we do not assume that this encoding is a functional relation. For
  example, we will need to encode pre-Kleene algebra terms as numbers.  Such a
  term is an equivalence class of syntax trees quotiented by provable equality.
  Thus, each term can be encoded as multiple natural numbers, one for each
  syntax tree in its equivalence class.  Nevertheless, by abuse of notation,
  we'll use the encoding notation as if it denoted a unique number.}

\begin{theoremEnd}{theorem}
  \label{thm:basic-inseparability}
  The following two languages are effectively inseparable:
  \begin{align*}
    A &  \triangleq  \{  \langle  M ,x \rangle   \mid  \text{The two-counter machine $M$ halts on $x$ and outputs 1}\} \\
    B &  \triangleq  \{  \langle  M ,x \rangle   \mid  \text{The two-counter machine $M$ halts on $x$ and outputs 0}\}.
  \end{align*}
  In other words, there is a partial computable function $f$ with the following
  property.  Given a machine $M$, let $W_M$ be the set of inputs accepted by
  $M$.  Suppose that $M_1$ and $M_0$ are such that
  $W_{M_1} \cap W_{M_0} = \emptyset$, $A \subseteq W_{M_1}$ and
  $B \subseteq W_{M_0}$.  Then $f\langle M_1, M_0 \rangle$ is defined and does
  not belong to $W_{M_1} \cup W_{M_0}$.
\end{theoremEnd}

\begin{proofEnd}
  We implement $f$ as follows. Given an input $x$, if $x$ does not encode a pair
  of machines, then the output is undefined.  Otherwise, suppose that
  $x = \langle M_1, M_0 \rangle$.  Construct a machine $M_\eta$ as follows.  On
  an input $x$, run $M_1$ and $M_0$ on $\langle x, x \rangle$ in parallel.  If
  $M_i$ accepts first, then halt and output $1 - i$. If neither accept, then
  just run forever.  We pose $f(x) = \langle M_\eta, M_\eta \rangle$.

  We need to show that $f(x) \notin W_{M_1} \cup W_{M_0}$ when
  $x = \langle M_1, M_0\rangle$ and the two machines satisfy the above
  hypotheses.  Suppose that $f(x) = \langle M_\eta, M_\eta \rangle \in W_{M_1}$.
  By the definition of $M_\eta$, this means that $M_\eta$ outputs 0 on
  $\langle M_\eta \rangle$.  Thus
  $\langle M_\eta, M_\eta \rangle \in B \subseteq W_{M_0}$. This contradicts the
  hypothesis that $W_{M_1} \cap W_{M_0} = \emptyset$.  Thus,
  $f(x) \notin W_{M_1}$.  An analogous reasoning shows that
  $f(x) \notin W_{M_0}$, which allows us to conclude.
\end{proofEnd}

Effective inseparability is a strengthening of the notion of inseparability,
which says that two sets cannot be distinguished by a total computable
function.  If we are just interested in the undecidability of the equational
theory, then basic inseparability is enough, as the following argument shows:

\begin{theorem}[Undecidability]
  \label{thm:undecidability}
  Let $ \Sigma   \triangleq  \{0,1\}$ be a discrete commutable set.  Suppose that we have a
  diagram of sets
  \begin{center}
    \begin{tikzcd}
      \mathcal{T}\ddot{ \Sigma } \ar[r, "l"] \ar[dr, "l'", swap] & \mathcal{L}\ddot{ \Sigma } \\
      & X, \ar[u]
    \end{tikzcd}
  \end{center}
  where $l'$ is computable.  Then equality on $X$ is undecidable.  In
  particular, equality is undecidable on $\mathcal{T}\ddot{ \Sigma }$,
  $\mathcal{K}\ddot{ \Sigma }$ and $\mathcal{L}\ddot{ \Sigma }$.
\end{theorem}

\begin{proof}
  Let $A$ and $B$ be the sets of \Cref{thm:basic-inseparability}.  Let's define
  a computable function $\eta : \Sigma^* \to \Sigma^*$ with the following
  properties:
  \begin{itemize}
  \item if $s \in A$, then $\eta(s) \in X_{=}$, where
    $X_{=} \triangleq \{\langle x, y \rangle \mid \text{$x$ and $y$ encode the
      same element of $X$}\}$.
  \item if $s \in B$, then $\eta(s) \notin X_{=}$.
  \end{itemize}
  Find a suitable encoding of the characters of $ \Sigma_M$ as binary strings,
  which leads to the following injective embeddings:
  \begin{center}
    \begin{tikzcd}
      \mathcal{T}\ddot{ \Sigma }_M \ar[r, "l"] \ar[d, hook]
      & \mathcal{L}\ddot{ \Sigma }_M \ar[d,hook] \\
      \mathcal{T}\ddot{ \Sigma } \ar[r, "l"] \ar[dr, "l'", swap] & \mathcal{L}\ddot{ \Sigma } \\
      & X. \ar[u]
    \end{tikzcd}
  \end{center}
  In what follows, we'll treat $\mathcal{T}\ddot{\Sigma}_M$ and
  $\mathcal{L}\ddot{\Sigma}_M$ as subsets of $\mathcal{T}\ddot{\Sigma}$ and
  $\mathcal{L}\ddot{\Sigma}$, to simplify the notation.

  Suppose that we are given some string $s \in \Sigma^*$.  We define $\eta(s)$
  as follows.  We can assume that $s$ is of the form $ \langle M, n \rangle$,
  where $M$ is a machine and $n \in \mathbb{N}$ (if $s$ is not of this form, we
  define the output as $\eta(s) = \langle l'(0), l'(1)\rangle $).  First, we
  compute the term $\rho$ of \Cref{thm:completeness}, using $a^nb^0\dot{q}$ as
  the initial configuration.  Next, let $e_L$ and $e_R$ be the left- and
  right-hand sides of the inequality of \Cref{thm:completeness}.  We pose
  \( \eta\langle M , n \rangle \triangleq \langle l'(e_L+e_R),l'(e_R)\rangle.
  \)

  If $\langle M, n \rangle \in A$, the inequality of \Cref{thm:completeness} is
  valid.  Thus, $e_L \le e_R$ holds, or, equivalently, $e_L + e_R = e_R$.  Thus,
  $l'(e_L+e_R) = l'(e_R)$ is valid, which implies that $\eta(s) \in X_{=}$.

  If, on the other hand, $M$ outputs $0$ on $n$ (that is,
  $\langle M, n \rangle \in B$), we claim that $\eta(s) \notin X_{=}$.  It
  suffices to prove $l'(e_L+e_R) \neq l'(e_R)$. Aiming for a contradiction,
  suppose that $l'(e_L+e_R) = l'(e_R)$. This implies
  $l(e_L+e_R) = l(e_L) + l(e_R) = l(e_R)$, which is equivalent to the inequality
  $l(e_L) \le l(e_R)$.  Let $e_R' \in \mathcal{T}\ddot{\Sigma}_M$ be the
  right-hand side of the inequality of \Cref{thm:soundness}.  We have
  $l(e_R) \le l(e_R')$ because $l(\rho) \le l(\Sigma_M^*)$.  Thus,
  $l(e_L) \le l(e_R')$.  However, by \Cref{thm:soundness}, this can only hold if
  $M$ outputs $1$, which contradicts our assumption.

  To conclude, suppose that $d : \Sigma^* \to \{0,1\}$ is a decider for $X_{=}$
  (that is, suppose that equality on $X$ is decidable).  Then $d \circ \eta$ can
  separate the sets $A$ and $B$, which contradicts
  \Cref{thm:basic-inseparability} because two effectively inseparable sets are
  also computationally inseparable. Therefore, such a $d$ cannot exist.
\end{proof}

However, if we also want a more precise characterization of the complexity of
this theory, the notion of effective inseparability is crucial.  The following
argument, which refines the previous proof, is based on Kuznetsov's
work~\cite{Kuznetsov23}.

\begin{theorem}[Complexity]
  If equalities in \(X\) (from~\cref{thm:undecidability}) are recursive
  enumerable, then equalities in \(X\) are \(\Sigma^0_1\)-complete.  In
  particular, equalities in $\mathcal{T}\ddot{ \Sigma }$ and
  $\mathcal{K}\ddot{ \Sigma }$ are \(\Sigma^0_1\)-complete.
\end{theorem}

\begin{proof}
  Let $A' \triangleq \{ \langle M, n \rangle \mid l'(e_L+e_R) = l'(e_R) \}$,
  where $(M,n)$ is a machine-input pair and $e_L$ and $e_R$ are defined as in
  the proof of~\cref{thm:undecidability}.  By arguments in the proof
  of~\cref{thm:undecidability}, if $\langle M, n\rangle \in A$, then
  \(l'(e_L+e_R) = l'(e_R)\), which means \(A' \supseteq A\).  By similar
  arguments, $A' \cap B = \emptyset$.

  By folklore~\cite[Proposition 9]{Kuznetsov23}, \(A'\) and \(B\) are
  effectively inseparable because $A$ and $B$ are.  Moreover, note that both
  \(A'\) and \(B\) are in \(\Sigma^0_1\) --- membership in \(A'\) is recursively
  enumerable because we can compute $e_L$ and $e_R$ from $(M,n)$ and enumerate
  the possible proofs of $l'(e_L+e_R) = l'(e_R)$.  This implies that \(A'\) is
  \(\Sigma^0_1\)-complete~\cite[Proposition 7]{Kuznetsov23}, and therefore
  \(\Sigma^0_1\)-hard.

  The function $\eta$ in the proof of \Cref{thm:undecidability} has the property
  that $\eta(s) \in X_{=}$ if and only if $s \in A'$. (This relies on the fact
  that we defined $\eta(s) = \langle l'(0), l'(1)\rangle$ when $s$ is not the
  encoding of a machine-input pair, and that $l'(0) \neq l'(1)$ because
  $l(0) \neq l(1)$).  In other words, $\eta$ is a reduction from $A'$ to
  $X_{=}$, which proves $X_{=}$ is $\Sigma^0_1$-hard.  We conclude because we
  assumed that equality on $X$ is in $\Sigma^0_1$.
\end{proof}

Thus, to establish undecidability, we need to prove soundness and completeness.
The easiest part is proving soundness: we just need to adapt the proof of
undecidability of equations of $*$-continuous Kleene algebras with commutativity
conditions~\cite{Kozen_1996}.
For completeness, however, we need to do some more work. Roughly speaking, we
first prove that $R_M$ satisfies an analogue of the completeness theorem for a
single transition, and then show that this version implies a more general one
for an arbitrary number of transitions (\Cref{sec:relations}).

The main challenge for proving the single-step version of completeness is that
we can no longer leverage properties of regular languages, and must reason
solely using the laws of pre-Kleene algebra.  Our strategy is to show that $R_M$
is just as good as its corresponding regular language if we want to reason about
\emph{prefixes} of matched strings.  Given any string $s' \leq R_M$ and a
current state $s$, we can tell whether $s'$ encodes a valid sequence of
transitions or not simply by looking at some finite prefix determined by $s$.
This finite prefix can be extracted by unfolding $R_M$ finitely many times,
which can be done in the setting of preKA.

\section{Representing Relations}
\label{sec:relations}

In this section, we show that we can reduce the statement of completeness to a
similar statement about single transitions.  If $e  \in  \mathcal{T}\ddot{ \Sigma }$ and
$ \Lambda   \subseteq  \mathcal{S} \Sigma $ is a set of strings, we write $\Next_e( \Lambda )$ to denote the
image of $ \Lambda $ by $ \to _e$; that is, the set $ \bigcup _{s  \in   \Lambda } \{ s'  \mid  s  \to _e s' \}$.

\begin{definition}
  \label{def:representable-relation}
  Let $L  \in  \mathcal{T} \Sigma $ be term.  We say that a term $e  \in  \mathcal{T}\ddot{ \Sigma }$
  is a \emph{representable relation} on $L$ if the following conditions hold:
  \begin{itemize}
  \item $ \pi _l(e)  \leq  L$;
  \item $ \pi _r(e)  \leq  L$;
  \item $\Next_e( \Lambda )$ is finite if $ \Lambda $ is (note that we must have
    $\Next_e( \Lambda )  \leq   \pi _r(e)  \leq  L$);
  \item there exists some \emph{residue term} $ \rho $ such that
    $ \Lambda _re  \leq   \Lambda \Next_e( \Lambda )_r +  \Sigma ^* \Sigma ^{ \neq } \rho $ for every finite $ \Lambda $.
  \end{itemize}
  We write $e : \Rel(L)$ to denote the type of $e$.
\end{definition}

Given a representable relation $e$, we can iterate the above inequality several
times when reasoning about its reflexive transitive closure $e^*$:

\begin{theoremEnd}{lemma}
  \label{lem:representable-star}
  Suppose that $e : \Rel(L)$.  There exists some $ \rho $ such that, for every
  $n  \in  \mathbb{N}$ and every finite $ \Lambda   \leq  L$, we have the inequality
  \(  \Lambda _re^*  \leq   \Sigma ^*\Next_e^{<n}( \Lambda )_r +  \Sigma ^*\Next_e^n( \Lambda )_re^* +  \Sigma ^* \Sigma ^{ \neq } \rho , \) where
  $\Next_e^{<n} =  \bigcup _{i < n}\Next_e^i( \Lambda )$.
\end{theoremEnd}

\begin{proofEnd} %
  Let $ \rho   \triangleq   \rho 'e^*$, where $ \rho '$ is the residue of $e$. Abbreviate $ \Sigma ^* \Sigma ^{ \neq } \rho $ as
  $ \varepsilon $.  We proceed by induction on $n$.  If $n = 0$, then the goal becomes
  $ \Lambda _r e^*  \leq   \Sigma ^*\Next_e^0( \Lambda )_re^* +  \varepsilon $, which holds because $\Next_e^0( \Lambda ) =  \Lambda $.

  Otherwise, for the inductive step, suppose that the goal is valid for $n$.  We
  need to prove that it is valid for $n + 1$.  Recall that
  $ \Lambda '  \triangleq  \Next_e( \Lambda )  \leq  L$. We have
  \begin{align*}
    &  \Lambda _re^* \\
    & =  \Lambda _r +  \Lambda _ree^* \\
    &  \leq   \Lambda _r +  \Lambda \Next_e( \Lambda )_re^* +  \varepsilon 
    & \text{($e$ is representable)} \\
    & =  \Lambda _r +  \Lambda  \Lambda '_re^* +  \varepsilon  \\
    &  \leq   \Lambda _r %
      +  \Lambda \left( \Sigma ^*\Next_e^{<n}( \Lambda ')_r+ \Sigma ^*\Next_e^n( \Lambda ')_re^* +  \varepsilon \right) +  \varepsilon 
    & \text{I.H.} \\
    & =  \Lambda _r +  \Lambda  \Sigma ^*\Next_e^{<n}( \Lambda ')_r+ \Lambda  \Sigma ^*\Next_e^n( \Lambda ')_re^* +  \Lambda  \varepsilon  +  \varepsilon  \\
    &  \leq   \Sigma ^* \Lambda _r +  \Sigma ^*\Next_e^{<n}( \Lambda ')_r+ \Sigma ^*\Next_e^{n}( \Lambda ')_re^* +  \varepsilon  +  \varepsilon 
    & \text{($ \Lambda $ is finite)} \\
    & =  \Sigma ^*\Next_e^0( \Lambda )_r +  \Sigma ^*\Next_e^{<n}( \Lambda ')_r+ \Sigma ^*\Next_e^{n}( \Lambda ')_re^* +  \varepsilon  \\
    & =  \Sigma ^*\Next_e^{<n+1}( \Lambda )_r +  \Sigma ^*\Next_e^{n+1}( \Lambda )_re^* +  \varepsilon . & \qedhere
  \end{align*}
\end{proofEnd}

If we know that the number of transitions from a given set of initial states is
bounded, we obtain the following result.

\begin{theoremEnd}{theorem}
  \label{thm:representable-star-finite}
  If $e : \Rel(L)$, there exists $ \rho $ such that, given $n  \in  \mathbb{N}$ and a finite
  $ \Lambda   \leq  L$, if $\Next_e^n( \Lambda ) =  \emptyset $, then
  \(  \Lambda _re^*  \leq   \Sigma ^*\Next_e^{< n}( \Lambda )_r +  \Sigma ^* \Sigma ^{ \neq } \rho . \)
\end{theoremEnd}

\begin{proofEnd} %
    Choose the same $ \rho $ as in \Cref{lem:representable-star}.  Then
  \begin{align*}
    &  \Lambda _re^* \\
    &  \leq   \Sigma ^*\Next_e^{<n}( \Lambda )_r +  \Sigma ^*\Next_e^n( \Lambda )_re^* +  \Sigma ^* \Sigma ^{ \neq } \rho 
    & \text{by \Cref{lem:representable-star}} \\
    & =  \Sigma ^*\Next_e^{<n}( \Lambda )_r +  \Sigma ^* \Sigma ^{ \neq } \rho . & \qedhere
  \end{align*}
\end{proofEnd}

\section{Proving Representability}
\label{sec:proving-representability}

In this section, we prove that the transition relation $R_M$ of a two-counter
machine is a representable relation, which will allow us to derive completeness
from \Cref{thm:representable-star-finite}.  To do this, we need to show how we
can use finite unfoldings of a relation to pinpoint certain terms that
definitely match the ``error'' term $ \Sigma _M^* \Sigma _M^{ \neq } \rho $.

\subsection{Automata theory}
\label{sec:automata}

One of the pleasant consequences of working with Kleene algebra is that many
intuitions about regular languages carry over.  In particular, we can analyze
terms by characterizing them as automata.  This can be done algebraically by
posing certain \emph{derivative operations} $ \delta _x$ on terms, which satisfy
a \emph{fundamental theorem}~\cite{Silva_2010}: given a term
$e \in \mathcal{K}X$, we have $e = e_0 + \sum _{x \in X} x \cdot \delta _x(e)$,
where $e_0 \in \{0,1\}$.  Intuitively, each term in this equation corresponds to
a state of some automaton.  The term $e$ corresponds to the starting state of
the automaton, the null term $e_0$ states whether the starting state is
accepting, and each $ \delta _x(e)$ the state we transition too after observing
the character $x \in X$.  Derivatives can be iterated, describing the behavior
of the automaton as it reads larger and larger strings, and which of those
strings are accepted by it.  This would be useful for our purposes, because such
iterated derivatives would allow us to compute all prefixes up to a given length
that can match an expression.  Unfortunately, this theory of derivatives hinges
on the induction properties of Kleene algebra, and it is unlikely that it can be
adapted in all generality to the preKA setting.  For example, the closest we can
get to an expansion for $1^*$ is $1^* = 1 + 1 \cdot 1^* = 1 + 1^*$, which does
not have the required form.
Indeed, as demonstrated by~\cref{exp:preKA-that-is-not-a-WKA}, 
the star operation no longer preserves the multiplicative identity in preKA.

To remedy this issue, we are going to carve out a set of so-called
\emph{finite-state terms} of a pre-Kleene algebra, for which this type of
reasoning is sound.  Luckily, most regular operations preserve finite-state
terms; we just need to be a little bit careful with the star operation.  We
start by defining \emph{derivable} terms, which can be derived at least once.
Finite-state terms will then allow us to iterate derivatives.

\begin{definition}
  Let $e  \in  \mathcal{T}X$ be a term, where $X$ is finite.  We say that $e$ is
  \emph{derivable} if there exists a family of terms $\{ \delta _x(e)\}_{x  \in  X}$ such
  that $e = [e]_0 +  \sum _x x \delta _x(e)$.  
  Recall \([-]_0\) is the homomorphism \([-]_0: \mathcal{T}X  \to  \mathbb{2}\) such that \([e]_0 = 1  \iff  e  \geq  1\).
  We refer to the term $ \delta _x(e)$ as the
  \emph{derivative with respect to $x$}.
\end{definition}

The family $ \delta _x(e)$ is not necessarily unique.  Nevertheless, we'll use the
notation $ \delta _x(e)$ to refer to specific derivatives of $x$ when it is clear from
the context which one we mean.

\begin{theoremEnd}{lemma}
  \label{lem:derivable-closure}
  Derivable terms are closed under all the pre-Kleene algebra operations, with
  the following caveats: for $e^*$, we also require that $[e]_0 = 0$; for
  $e_1e_2$, the term is also derivable if $e_2$ isn't, provided that
  $[e_1]_0=0$.  We have the following choices of derivatives:
  \begin{align*}
     \delta _x(0) & = 0 &
     \delta _x(1) & = 0 \\
     \delta _x(x) & = 1 &
     \delta _x(y) & = 0 \quad \text{if $y  \neq  x$} \\
     \delta _x(e_1 + e_2) & =  \delta _x(e_1) +  \delta _x(e_2) &
     \delta _x(e_1e_2) & = [e_1]_0 \delta _x(e_2) +  \delta _x(e_1)e_2 \\
     \delta _x(e^*) & =  \delta _x(e)e^*,
  \end{align*}
  where, by abuse of notation, we treat $[e_1]_0 \delta _x(e_2)$ as 0 when $e_2$ is not
  necessarily derivable (since, by assumption, $[e_1]_0 = 0$ in that case).
\end{theoremEnd}

\begin{proofEnd} %
  We prove the closure property for products and star. For products, we start by
  expanding $e_1$:
  \begin{align*}
    e_1e_2
    & = \left([e_1]_0 +  \sum _x x \delta _x(e_1)\right)e_2 \\
    & = [e_1]_0e_2 +  \sum _x x \delta _x(e_1)e_2.
  \end{align*}

  If $[e_1]_0=0$, the first term gets canceled out, and we obtain
  $ \sum _xx \delta _x(e_1)e_2 = [e_1]_0[e_2]_0 +  \sum _xx \delta _x(e_1)e_2$.  Otherwise, we know that
  $e_2$ is derivable, and we proceed as follows:
  \begin{align*}
    e_1e_2
    & = [e_1]_0\left([e_2]_0 +  \sum _x x  \delta _x(e_2)\right) +  \sum _x x \delta _x(e_1)e_2 \\
    & = [e_1]_0[e_2]_0 +  \sum _x [e_1]_0 x  \delta _x(e_2) +  \sum _x x \delta _x(e_1)e_2 \\
    & = [e_1]_0[e_2]_0 +  \sum _x x ([e_1]_0 \delta _x(e_2) +  \delta _x(e_1)e_2)
    & (\text{because $[e_1]_0x = x[e_1]_0$}),
  \end{align*}
  which allows us to conclude.

  For star, assuming that $[e]_0 = 0$, we note that $e^* = 1 + ee^*$, and we
  apply the closure properties for the other operations. %
\end{proofEnd}

\begin{definition}
  \label{def:finite-state}
  Suppose that $X$ is finite.  A \emph{finite-state automaton} is a finite set
  $S$ of elements of $\mathcal{T}X$ (the \emph{states}) that contains 1, is
  closed under finite sums and under derivatives (that is, every $e \in S$ is
  derivable, and each $ \delta _x(e)$ is a state).  We say that a term $e$ is
  \emph{finite state} if it is a state of some finite-state automaton $S$.
\end{definition}

Requiring that the states of an automaton be closed under sums means, roughly
speaking, that we are working with non-deterministic rather than deterministic
automata, generalizing the notion of Antimirov's derivative~\cite{antimirov_PartialDerivativesRegular_1996}.
This treatment is convenient for the commutative setting, since a given string
could be matched by choosing different orderings of its characters.

Finite-state terms can, in fact, be inductively constructed from the operations
of pre-Kleene algebra, thus making the identification of a finite-state term
trivial.

\begin{theoremEnd}{lemma}
  \label{lem:finite-state-closure}
  Let $X$ be a finite commutable set. Finite-state terms are preserved by all
  the pre-Kleene algebra operations (for $e^*$, we additionally require that
  $[e]_0 = 0$).  Moreover, the set of states of the corresponding automata can
  be effectively computed.
\end{theoremEnd}

\begin{proofEnd} %
  Let's consider all the cases.
  \begin{itemize}
  \item The set $\{0,1\}$ is an automaton by \Cref{lem:derivable-closure}.
    Therefore, $0$ and $1$ are finite state.
  \item By \Cref{lem:derivable-closure}, if $x$ is a symbol, the set $S = \{x\}$
    is a pre-automaton.  Therefore, $x$ is finite state because it belongs to
    the automaton $\bar{S}$.
  \item Suppose that $S_1$ and $S_2$ are finite automata.  By
    \Cref{lem:derivable-closure}, the set
    $S = \{e_1 + e_2  \mid  e_1  \in  S_1, e_2  \in  S_2\}$ is a pre-automaton.  Therefore,
    if we have finite-state terms $e_1$ and $e_2$ of $S_1$ and $S_2$, their sum
    $e_1+e_2$ is finite state because it belongs to the automaton $\bar{S}$.
  \item Suppose that $S_1$ and $S_2$ are finite automata.  By
    \Cref{lem:derivable-closure}, the set
    $S = \{e_1e_2 \mid  e_1  \in  S_1, e_2  \in  S_2\}$ is a pre-automaton.  Indeed,
    $ \delta _x(e_1e_2) = [e_1]_0 \delta _x(e_2) +  \delta _x(e_1)e_2$ is a sum of elements of $S$,
    since
    \begin{align*}
      [e_1]_0 &  \in  S_1 \\
       \delta _x(e_2) &  \in  S_2 \\
       \delta _x(e_1) &  \in  S_1 \\
      e_2 &  \in  S_2.
    \end{align*}
    Therefore, if we have finite-state terms $e_1$ and $e_2$ of $S_1$ and $S_2$,
    their product $e_1e_2$ is finite state because it belongs to the automaton
    $\bar{S}$.
  \item Suppose that $e$ is a state of some automaton $S$ such that $[e]_0 = 0$.
    Define $S' = \{ e'e^*  \mid  e'  \in  S \}$.  By \Cref{lem:derivable-closure}, this
    set is a pre-automaton.  Indeed,
    \begin{align*}
       \delta _x(e'e^*)
      & = [e']_0 \delta _x(e^*) +  \delta _x(e')e^* \\
      & = [e']_0 \delta _x(e)e^* +  \delta _x(e')e^* \\
      & = ([e']_0 \delta _x(e)+ \delta _x(e'))e^*.
    \end{align*}
    The terms $ \delta _x(e)$ and $ \delta _x(e')$ are in $S$. Thus, $[e']_0 \delta _x(e)  \in  S$ and
    $ \delta _x(e'e^*)$ is a sum of terms of $S'$.  Since $e^* = 1e^*$ is an element of
    $S'$, then it is a state of $\bar{S}'$, and $e^*$ is finite state. \qedhere
  \end{itemize}
\end{proofEnd}

Furthermore, since terms in a finite-state automaton are closed under
derivatives, we can unfold them via derivatives \(k\) times.  This unfolding
will turn a term into a sum of some strings that are shorter than \(k\); and
some strings \(s\) with length exact \(k\), followed the residual expressions
\(e_{s}\) indexed by \(s\).  Formally, we can express this property as follows.

\begin{theoremEnd}{lemma}
  \label{lem:expansion}
  Let $e  \in  \mathcal{T}X$ be a state of a finite-state automaton $S$, and
  $k  \in  \mathbb{N}$.  We can write
  \[ e =  \sum \{s \mid s  \in  \mathcal{S}X, s  \leq  e, |s| < k\} + %
     \sum \{se_s \mid s  \in  \mathcal{S}X, |s| = k\}, \] where each $e_s  \in  S$ for
  all $s$, and the size $|s|  \in  \mathbb{N}$ of a string $s$ is defined by mapping every
  symbol of $s$ to $1  \in  \mathbb{N}$.
\end{theoremEnd}

\begin{proofEnd} %
  By induction on $k$.  When $k = 0$, the equation is equivalent to $e = e$, and
  we are done.  Otherwise, suppose that the result is valid for $k$.  We need to
  prove that it is also valid for $k+1$.  Write
  \[ e =  \sum _{\substack{s  \in  \mathcal{S}X \\ s  \leq  e \\ |s| < k}} s + %
     \sum _{\substack{ s  \in  \mathcal{S}X \\ |s| = k}} se_s. \] By deriving each $e_s$, we can
  rewrite this as
  \begin{align}
    e
    & =  \sum _{\substack{s  \in  \mathcal{S}X \\ s \leq  e \\ |s| < k}} s + %
     \sum _{\substack{s  \in  \mathcal{S}X \\ |s| = k}}s\left([e_s]_0 +  \sum _{x  \in  X} x \delta _x(e_s)\right) \nonumber \\
    & =  \sum _{\substack{s  \in  \mathcal{S}X \\ s  \leq  e \\ |s| < k}} s + %
     \sum _{\substack{s  \in  \mathcal{S}X \\ |s| = k}} s[e_s]_0 + %
     \sum _{\substack{s  \in  \mathcal{S}X \\ |s|=k}} \sum _{x  \in  X}sx \delta _x(e_s).
    \label{eq:expansion-1}
  \end{align}
  We can see that $[e_s]_0 = 1$ if and only if $s  \leq  e$: by taking the language
  interpretation of \labelcref{eq:expansion-1}, we can see that a string of size
  $k$ can only belong to the middle term, since the left and right terms can
  only account for strings of strictly smaller or larger size, respectively.
  Thus, we can rewrite \labelcref{eq:expansion-1} as
  \begin{align}
    e
    & =  \sum _{\substack{s  \in  \mathcal{S}X \\ s  \leq  e \\ |s| < k}} s + %
     \sum _{\substack{s  \in  \mathcal{S}X \\ |s| = k \\ s  \leq  e}} s[e_s]_0 + %
     \sum _{s, |s|=k} \sum _{x  \in  X}sx \delta _x(e_s) \nonumber \\
    & =  \sum _{\substack{s  \in  \mathcal{S}X \\ s  \leq  e \\ |s| < k+1}} s + %
     \sum _{\substack{s  \in  \mathcal{S}X \\ |s|=k}} \sum _{x  \in  X}sx \delta _x(e_s).
    \label{eq:expansion-2}
  \end{align}

  Given some string $s$ with $|s| = k+1$, define
  \begin{align*}
    e'_s
    &  \triangleq 
       \sum _{\substack{(s', x)  \in  \mathcal{S}X  \times  X \\ s = s'x}}  \delta _x(e_{s'}).
  \end{align*}
  This sum is well defined because there are only finitely many $s'$ and $x  \in  X$
  such that $s = s'x$: $s'$ must be of size $k$, and there are only finitely
  many such strings.  Moreover, $e'_s$ is an element of $S$, since $S$ is closed
  under taking derivatives and finite sums. We have
  \begin{align*}
    se'_s
    & =  \sum _{\substack{(s',x) \\ |s'| = k \\ s = s'x}} s \delta _x(e_{s'}) \\
    & =  \sum _{\substack{(s',x) \\ |s'| = k \\ s = s'x}} s'x \delta _x(e_{s'}).
  \end{align*}
  Therefore,
  \begin{align*}
     \sum _{\substack{s \\ |s|=k+1}} se'_s
    & =  \sum _{\substack{s \\ |s|=k+1}} \sum _{\substack{(s',x) \\ |s'| = k\\ s = s'x}} s'x \delta _x(e_{s'}) \\
    & =  \sum _{\substack{(s',x) \\ |s'| = k}} s'x \delta _x(e_{s'}) \\
    & =  \sum _{\substack{s' \\ |s'| = k}} \sum _{x  \in  X} s'x \delta _x(e_{s'}).
  \end{align*}

  Putting everything together, \labelcref{eq:expansion-2} becomes
  \begin{align}
    e & =  \sum _{\substack{s  \leq  e, |s| < k+1}} s +  \sum _{\substack{s \\ |s|=k+1}}se'_s
    \label{eq:expansion-3},
  \end{align}
  which completes the inductive case.
\end{proofEnd}

\subsection{Bounded-Output Terms}
\label{sec:bounded-output}

\Cref{lem:expansion} gives us almost what we need to prove that the transition
term $R_M$ is a representable relation. It allows us to partition \(R_M\) into
strings $s$ of length bounded by $k$ and terms of the form $se_s$, which match
strings prefixed by $s$ of length greater than $k$.  The first component, the
strings $s$, can be easily shown to satisfy the upper bound required for being
representable. However, the prefixes $s$ that appear in the terms $se_s$ are
arbitrary, and, since we are working with pre-Kleene algebra, there isn't much
we can leverage to show that such prefixes will yield a similar bound.  The
issue is that, in principle, in order to tell whether
$s'_rse_s  \leq   \Sigma _M^* \Sigma _M^{ \neq } \rho $, we might need to unfold $e_s$ arbitrarily deep,
which we cannot do in the preKA setting.  To rule out these issues, we introduce
a notion of \emph{bounded-output terms}, which guarantee that only a finite
amount of unfolding is necessary.

\begin{definition}
  \label{def:bounded-output}
  Let $e  \in  \mathcal{T}\ddot{ \Sigma }$ be a term.  We say that $e$ has \emph{bounded
    output} if there exists some $k  \in  \mathbb{N}$ (the \emph{fanout}) such that, for
  every string $s  \leq  e$, $| \pi _r(s)|  \leq  (| \pi _l(s)|+1)k$.
\end{definition}

\begin{theoremEnd}{lemma}
  \label{lem:step-finite}
  Let $e$ have bounded output with fanout $k$ and let $ \Lambda $ be finite.  If
  $s  \in  \Next_e( \Lambda )$, then $|s|  \leq  (m+1)k$, where $m = \max\{ |s'|  \mid  s'  \in   \Lambda \}$.
  Thus, since $ \Sigma $ is finite, $\Next_e( \Lambda )$ is finite.
\end{theoremEnd}

\begin{proofEnd}
  If $s  \in  \Next_e( \Lambda )$, by definition, there exists $s'  \in   \Lambda $ such that
  $s'_ls_r  \leq  e$.  Since $e$ has fanout $k$, we have
  \[ |s| = | \pi _r(s'_ls_r)|  \leq  (| \pi _l(s'_ls_r)|+1)k = (|s|+1)k  \leq  (n+1)k. \qedhere\]
\end{proofEnd}

\begin{theoremEnd}{lemma}
  \label{lem:bounded-output-closure}
  Bounded-output terms are closed under all the pre-Kleene algebra
  operations. For $e^*$, we additionally require that $| \pi _l(s)|  \geq  1$ for all
  strings $s  \leq  e$.
\end{theoremEnd}

\begin{proofEnd}
  Let's focus on the last point.  Suppose that $e$ has fanout $k$ and that
  $| \pi _l(s)|  \geq  1$ for every $s  \leq  e$.  We are going to show that $e^*$ has bounded
  output with fanout $2k$.

  Suppose that $s  \leq  e^*$.  We can write $s = s_1  \cdots  s_n$ such that $s_i  \leq  e$ for
  every $i  \in  \{1,  \ldots , n\}$.  We have, for every $i  \in  \{1, \ldots ,n\}$,
  $| \pi _r(s_i)|  \leq  (| \pi _l(s_i)| + 1)k.$ Thus,
  \begin{align*}
    | \pi _r(s)|
    & =  \sum _{i=1}^n | \pi _r(s_i)| \\
    &  \leq   \sum _{i=1}^n (| \pi _l(s_i)| + 1)k \\
    &  \leq   \sum _{i=1}^n 2| \pi _l(s_i)|k
    & (\text{because $| \pi _l(s_i)|  \geq  1$}) \\
    & = \left( \sum _{i=1}^n| \pi _l(s_i)|\right)2k \\
    & = | \pi _l(s_0)  \cdots   \pi _l(s_n)|2k\\
    & = | \pi _l(s_0  \cdots  s_n)|2k  \\
    & = | \pi _l(s)|2k \\
    &  \leq  (| \pi _l(s)| + 1)2k. & \qedhere
  \end{align*}
\end{proofEnd}

For bounded-output terms, we can improve the expansion of \Cref{lem:expansion}.

\begin{theoremEnd}{lemma}
  \label{lem:expansion-bounded-output}
  Let $e  \in  \mathcal{T}\ddot{ \Sigma }$ be a bounded-output term that is the state of
  some automaton $S$. There exists some $k  \in  \mathbb{N}$ such that $e$ has fanout $k$ and
  such that, for every $n  \in  \mathbb{N}$, we can write
  \begin{align*}
    e
    & =  \sum \{s \mid s  \leq  e, |s| < n\} + %
     \sum \{se_s \mid s  \in  \mathcal{S}\ddot{ \Sigma }, |s|=n, | \pi _r(s)|  \leq  (| \pi _l(s)|+1)k\},
  \end{align*}
  where $e_s  \in  S$ for every $s$.
\end{theoremEnd}

\begin{proofEnd}
  Let $k_0$ be the fanout of $e$.  For each $e'  \in  S$ such that $e'  \neq  0$, choose
  some string $w_{e'}  \leq  e'$.  Define
  $m  \triangleq  \max\{| \pi _l(w_{e'})|  \mid  e'  \in  S, e'  \neq  0 \}$ and $k \triangleq (m+1)k_0$. Since
  $k  \geq  k_0$, we know that $e$ has fanout $k$.  Moreover, by
  \Cref{lem:expansion}, we have
  \begin{align*}
    e
    & =  \sum _{\substack{s  \leq  e\\ |s| < n}} s + %
     \sum _{\substack{s  \in  \mathcal{S}X \\ |s|=n}} se_s \\
    & =  \sum _{\substack{s  \leq  e \\ |s| < n}} s + %
     \sum _{\substack{s  \in  \mathcal{S}X \\ |s|=n \\ e_s  \neq  0}}se_s,
  \end{align*}
  where each $e_s$ is a state of $S$.  If $s$ is such that $|s|=n$ and
  $e_s  \neq  0$, we have $sw_{e_s}  \leq  e$.  Therefore,
  \begin{align*}
    | \pi _r(s)|
    &  \leq  | \pi _r(sw_{e_s})| \\
    &  \leq  (| \pi _l(sw_{e_s})|+1)k_0 \\
    & = (| \pi _l(s)|+| \pi _l(w_{e_s})|+1)k_0 \\
    &  \leq  (| \pi _l(s)|+m+1)k_0 \\
    &  \leq  (| \pi _l(s)|+1)(m+1)k_0 \\
    & = (| \pi _l(s)|+1)k.
  \end{align*}
  Thus,
  \begin{align*}
    e
    & =  \sum _{\substack{s  \leq  e \\ |s| < n}} s + %
     \sum _{\substack{s \\ |s|=n \\ e_s  \neq  0 \\ | \pi _r(s)| \leq (| \pi _l(s)|+1)k} }se_s \\
    & =  \sum _{\substack{s  \leq  e \\ |s| < n}} s + %
     \sum _{\substack{s \\ |s|=n \\ | \pi _r(s)| \leq (| \pi _l(s)|+1)k}} se_s. \qedhere
  \end{align*}
\end{proofEnd}

\begin{definition}
  \label{def:prefix-free}
  A term $L$ over $ \Sigma $ is \emph{prefix free} if for all strings $s_1  \leq  L$ and
  $s_2  \leq  L$, if $s_1$ is a prefix of $s_2$, then $s_1 = s_2$.
\end{definition}

\begin{theoremEnd}{lemma}[normal]
  \label{lem:mismatch}
  Let $s$ and $s'$ be two strings over $ \Sigma $ such that one is not a prefix of the
  other, or vice versa.  Then we can write $s = s_0xs_1$ and $s' = s_0x's_1'$
  with $x  \neq  x'$.  Thus, $s_rs'_l\ddot{ \Sigma }^*  \leq   \Sigma ^* \Sigma ^{ \neq }\ddot{ \Sigma }^*.$
\end{theoremEnd}

\begin{proofEnd}
  By induction on the length of $s$.
\end{proofEnd}

\begin{theoremEnd}{lemma}
  \label{lem:representable-sufficient}
  Suppose that $e  \in  \mathcal{T}\ddot{ \Sigma }$ is such that $ \pi _l(e)  \leq  L$ and
  $ \pi _r(e)  \leq  L$, where $L$ is prefix free.  Suppose, moreover, that $e$ is
  finite-state and has bounded output.  Then $e : \Rel(L)$.
\end{theoremEnd}

\begin{proofEnd}
  We have already seen that $\Next_e( \Lambda )$ is finite when $ \Lambda $ is
  (\Cref{lem:step-finite}).  Thus, we need to find some $ \rho $ such that, for every
  finite $ \Lambda $,
  \[  \Lambda _re  \leq   \Lambda \Next_e( \Lambda )_r +  \Sigma ^* \Sigma ^{ \neq } \rho . \] Define $ \rho   \triangleq  \ddot{ \Sigma }^* \rho _e$, where
  $ \rho _e$ is the greatest element of the automaton of $e$. It suffices to prove
  the result for the case $ \Lambda  = \{s\}$.  Indeed, if the result holds for
  singletons, we have
  \begin{align*}
     \Lambda _re
    & =  \sum _{s  \in   \Lambda } s_re \\
    &  \leq   \sum _{s  \in   \Lambda } s\Next_e(s)_r +  \Sigma ^* \Sigma ^{ \neq } \rho 
    & \text{by assumption} \\
    &  \leq   \sum _{s  \in   \Lambda }  \Lambda \Next_e(s)_r +  \Sigma ^* \Sigma ^{ \neq } \rho  \\
    & =  \Lambda  \sum _{s  \in   \Lambda }\Next_e(s) +  \Sigma ^* \Sigma ^{ \neq } \rho  \\
    & =  \Lambda \Next_e( \Lambda ) +  \Sigma ^* \Sigma ^{ \neq } \rho .  & \qedhere
  \end{align*}

  Let $k$ be the constant of \Cref{lem:expansion-bounded-output} for $e$,
  $n = |s|$, and let $p = (k+1)(n+1)$. Let
  \[ \ddot{ \Lambda }  \triangleq  \{s'  \in  \mathcal{S}\ddot{ \Sigma }  \mid  |s'| =p+1, | \pi _r(s')|  \leq  (| \pi _l(s')| + 1)k \}.\] By
  applying \Cref{lem:expansion-bounded-output} to $e$, we can write
  \begin{align*}
    e
    & =  \sum _{\substack{s'  \leq  e \\ |s'| < p+1}} s' + %
       \sum _{s'  \in  \ddot{ \Lambda }}s'e_{s'} \\
    & =  \sum _{s'  \leq  e, |s'|  \leq  p} s' + %
       \sum _{s'  \in  \ddot{ \Lambda }}s'e_{s'} \\
    & =  \sum _{\substack{s'  \leq  e \\ |s'|  \leq  p \\  \pi _l(s') = s}} s' + %
       \sum _{\substack{s'  \leq  e \\ |s'|  \leq  p \\  \pi _l(s')  \neq  s}} s' + %
       \sum _{s'  \in  \ddot{ \Lambda }} s' e_{s'},
  \end{align*}
  Thus, to prove the inequality, it suffices to prove
  \begin{align}
    s_r \sum _{\substack{s' \leq e \\ |s'|  \leq  p \\  \pi _l(s') = s}} s'
    & = s \Next_e(s)_r
      \label{eq:representable-goal-1} \\
    s_r \sum _{\substack{s' \leq e \\ |s'|  \leq  p \\  \pi _l(s')  \neq  s}} s'
    &  \leq   \Sigma ^* \Sigma ^{ \neq } \rho 
      \label{eq:representable-goal-2} \\
    s_r \sum _{s'  \in  \ddot{ \Lambda }}s'e_s'
    &  \leq   \Sigma ^* \Sigma ^{ \neq } \rho 
      \label{eq:representable-goal-3}.
  \end{align}

  Let us start with \labelcref{eq:representable-goal-1}.  Notice that, for any
  string $s'$ over $\ddot{ \Sigma }$, we have $s' =  \pi _l(s')_l \pi _r(s')_r$.  Therefore,
  there is a bijection between the set of indices $s'$ of the sum and the set of
  strings $\Next_e(s)$.  The bijection is given by
  \begin{align*}
    s' &  \mapsto   \pi _r(s')  \in  \Next_e(s) \\
    \Next_e(s)  \ni  s' &  \mapsto  s_ls'_r.
  \end{align*}
  To prove that this is a bijection, we must show that the inverse produces
  indeed a valid index.  Notice that, if $s'  \in  \Next_e(s)$, by
  \Cref{lem:step-finite}, we have $|s'|  \leq  (n+1)k$, and thus
  $|s_ls'_r| = |s|+|s'|  \leq  (n+1)(k+1) = p$.

  By reindexing the sum in \labelcref{eq:representable-goal-1} with this
  bijection, we have
  \begin{align*}
    s_r \sum _{\substack{s' \leq e \\ |s'|  \leq  p \\  \pi _l(s') = s}} s'
    & = s_r \sum _{s'  \in  \Next_e(s)} s_ls'_r \\
    & = s_rs_l \sum _{s'  \in  \Next_e(s)} s'_r \\
    & = s_rs_l\left( \sum _{s'  \in  \Next_e(s)} s'\right)_r \\
    & = s \Next_e(s)_r.
  \end{align*}

  Next, let us look at \labelcref{eq:representable-goal-2}.  Suppose that $s'$
  is such that $s'  \leq  e$ and $ \pi _l(s')  \neq  s$.  Since $L$ is prefix free, and
  $ \pi _l(s')  \leq  L$, \Cref{lem:mismatch} applied to $s$ and $s'$ yields
  \[ s_ls'  \leq   \Sigma ^* \Sigma ^{ \neq }\ddot{ \Sigma }^*  \leq   \Sigma ^* \Sigma ^{ \neq }\ddot{ \Sigma }^* \rho _e =  \Sigma ^* \Sigma ^{ \neq } \rho , \] where we
  use the fact that $ \rho _e  \geq  1$ because 1 is a state of the automaton of $e$.
  Summing over all such $s'$, we get the desired inequality.

  To conclude, we must show \labelcref{eq:representable-goal-3}.  By
  distributivity, this is equivalent to showing that, for every $s'  \in   \Lambda $,
  \begin{align*}
    s_r s' e_{s'}
    &  \leq   \Sigma ^* \Sigma ^{ \neq } \rho  \label{eq:representable-goal-3}.
  \end{align*}
  If $e_{s'} = 0$, we are done. Otherwise, by
  \Cref{thm:emptiness-characterization}, we can find some string $s''  \leq  e_{s'}$.
  We have $s's''  \leq  s'e_{s'}  \leq  e$.

  Note that we must have $| \pi _l(s')| > n$.  Indeed, suppose that $| \pi _l(s')|  \leq  n$.
  Since $s'  \in   \Lambda $, we have
  \begin{align*}
    |s'|
    & = | \pi _l(s')|+| \pi _r(s')| \\
    &  \leq  | \pi _l(s')|+(| \pi _l(s')|+1)k \\
    &  \leq  (| \pi _l(s')|+1)(k+1) \\
    &  \leq  (n+1)(k+1) \\
    & < p+1 \\
    & = |s'|,
  \end{align*}
  which is a contradiction.

  Since $ \pi _l(s's'')  \leq   \pi _l(e)  \leq  L$ and $L$ is prefix free, by
  \Cref{lem:mismatch}, we can write $s = s_0xs_1$ and
  $ \pi _l(s's'') =  \pi _l(s') \pi _l(s'') = s_0x's_1'$, with $x  \neq  x'$.  But
  $| \pi _l(s')| > n = |s|$ and $|s_0| < |s|$, thus $ \pi _l(s')$ must be of the form
  $s_0x's_2'$.  We find that $s_rs' = s_r \pi _l(s') \pi _r(s')  \leq   \Sigma ^* \Sigma ^{ \neq }\ddot{ \Sigma }^*$,
  and thus
  \[ s_rs'e_{s'}  \leq   \Sigma ^* \Sigma ^{ \neq }\ddot{ \Sigma }^*e_{s'}  \leq   \Sigma ^* \Sigma ^{ \neq }\ddot{ \Sigma }^* \rho _e =
     \Sigma ^* \Sigma ^{ \neq } \rho . \]
\end{proofEnd}

\subsection{Putting Everything Together}
\label{sec:finalizing}

To derive completeness for two-counter machines (\Cref{thm:completeness}), it
suffices to show that the hypotheses of \Cref{lem:representable-sufficient} are
satisfied.

\begin{theoremEnd}[normal]{lemma}
  \label{lem:transition-is-representable}
  We have the following properties:
  \begin{itemize}
  \item $T_M$ is prefix free.
  \item $ \pi _l(R_M)  \leq  C_M  \leq  T_M$.
  \item $ \pi _r(R_M)  \leq  T_M$.
  \item $R_M$ is finite state (\Cref{def:finite-state}).
  \item $R_M$ has bounded output (\Cref{def:bounded-output}).
  \end{itemize}
  Thus, by \Cref{lem:representable-sufficient}, the term $R_M$ is a
  representable relation of type $\Rel(T_M)$.
\end{theoremEnd}

\begin{proofEnd} %
\iffull
\else%
  To show that $R_M$ is finite state and had bounded output, we just appeal to
  the closure properties of such terms
  \Cref{lem:finite-state-closure,lem:bounded-output-closure}.  The rest is
  routine.
\fi
\end{proofEnd}

We can finally conclude with the proof of completeness, thus establishing
undecidability (\Cref{thm:undecidability}).

\begin{proof}[Proof of \Cref{thm:completeness}]
  If $s = s_0  \to _{R_M}  \cdots   \to _{R_M} s_n = c_1$, we can show that $\Next_e^i(s)$ is
  $\{s_i\}$ for $i  \leq  n$ and $ \emptyset $ when $i > n$, because the transition relation is
  deterministic and because $c_1$ does not transition.  Moreover, by
  \Cref{lem:transition-is-representable}, we have $s_i  \leq  C_M$ for every $i < n$
  (since $(s_i)_l(s_{i+1})_r  \leq  R_M$).

  Choose $ \rho $ as in \Cref{thm:representable-star-finite}. We have
  \begin{align*}
    sR_M^*
    &  \leq   \Sigma ^*\Next_e^{<n+1}(s)_r +  \Sigma ^* \Sigma ^{ \neq } \rho  \\
    & =  \Sigma ^*(\Next_e^{<n}(s) + \Next_e^n(s))_r +  \Sigma ^* \Sigma ^{ \neq } \rho  \\
    &  \leq   \Sigma ^*(C_M + c_1)_r +  \Sigma ^* \Sigma ^{ \neq } \rho . \qedhere
  \end{align*}
\end{proof}

\section{Conclusion and Related Work}
\label{sec:conclusion}

In his seminal work, Kozen~\cite{DBLP:conf/lics/Kozen97} established several
hardness and completeness results for variants of Kleene algebra.  He noted that
deciding equality in $*$-continuous Kleene algebras with commutativity
conditions on primitives was not possible---more precisely, the problem is
$ \Pi _1^0$-complete, by reduction from the complement of the Post
correspondence problem (PCP).  However, at the time, it was unknown whether a
similar result applied to the pure theory of Kleene algebra with commutativity
conditions ($\mathcal{K}X$).  The question had been left open since then. Our
work provides a solution, proving that the problem is undecidable, even for a
much weaker theory $\mathcal{T}X$, which omits the induction axioms of Kleene
algebra.

As we were about to post publicly this work, we became aware of the work of
Kuznetsov~\cite{Kuznetsov23}, who independently proved a similar result.  There
are two main differences between our results and his.  Originally, our proof
only established the undecidability of the theory of Kleene algebra with
commutativity conditions, whereas Kuznetsov's work proved its
$\Sigma^0_1$-completeness as well by leveraging the notion of \emph{effective
  inseparability}.  Since learning about his work, we managed to adapt his ideas
to our setting, thus obtaining completeness as well.  On the other hand,
Kuznetsov's proof requires the induction axiom of Kleene algebra to simplify
some of the inequalities involving starred terms---specifically, he needs the
identity $A^*(A^*)^+ \leq A^*$ and the monotonicity of $(-)^*$, whereas our
proof also applies to the weaker theory of pre-Kleene algebra.  In this sense,
we can view the results reported here as a synthesis of Kuznetsov's work and
ours.

In terms of techniques, both of our works draw inspiration from the proof of
$ \Pi _1^0$-completeness of the equational theory of $*$-continuous
KA. Leveraging the reduction of the halting problem to the PCP, Kuznetsov used
Kleene-algebra inequalities to describe \emph{self-looping} Turing
machines---that is, Turing machines that run forever by reaching a designated
configuration that steps to itself.  He then showed that the set of machine-input pairs
$ \langle M ,x \rangle $ where machines $M$ that reach a self-looping state on
input $x$ is recursively inseparable from the set of such pairs where $M$ halts
on the input, which implies that such inequalities cannot decidable.

The inequalities used by Kuznetsov are similar to ours, and can be proved by
unfolding finitely many times the starred term that defines the execution of
Turing machines, and by applying standard Kleene algebra inequalities that
follow from induction.  One important difference is that, in Kuznetsov's work,
this starred term contains only $*$-free terms, which arise from the reduction
of the halting problem to the PCP.  This requires some more work to establish
that the inequality indeed encodes the execution of the Turing machine, but this
work just replicates the ideas behind the standard reduction from the halting
problem to the PCP, so it does not need to be belabored. On the other hand, we
leverage the language of Kleene algebra to define an execution model for
two-counter machines, which can be encoded more easily.  The downside of our
approach is that our relation $R_M$ involves starred terms, which require our
notion of bounded output to be analyzed effectively.

\bibliographystyle{plainurl}
\bibliography{refs.bib}


\iffull \else
\appendix
\section{Detailed Proofs}
\printProofs
\fi

\end{document}